\documentclass[11pt]{amsart}
\usepackage[active]{srcltx}
\usepackage{calc,amssymb,amsthm,amsmath,amscd, ulem}
\usepackage{alltt}
\usepackage[left=1.35in,top=1.25in,right=1.35in,bottom=1.25in]{geometry}
\RequirePackage[dvipsnames,usenames]{color}

\DeclareFontFamily{OMS}{rsfs}{\skewchar\font'60}
\DeclareFontShape{OMS}{rsfs}{m}{n}{<-5>rsfs5 <5-7>rsfs7 <7->rsfs10 }{}
\DeclareSymbolFont{rsfs}{OMS}{rsfs}{m}{n}
\DeclareSymbolFontAlphabet{\scr}{rsfs}

\newtheorem{theorem}{Theorem}[section]

\newtheorem{proposition}[theorem]{Proposition}
\newtheorem{corollary}[theorem]{Corollary}

\theoremstyle{definition}
\newtheorem{definition}[theorem]{Definition}
\newtheorem{example}[theorem]{Example}

\theoremstyle{remark}
\newtheorem{remark}[theorem]{Remark}

\newtheorem{question}[theorem]{Question}

\numberwithin{equation}{theorem}

\newcommand{\ZZ}{\mathbb{Z}}

\DeclareMathOperator{\Spec}{{Spec}}

\newcommand{\sF}{\scr{F}}
\newcommand{\sD}{\scr{D}}
\DeclareMathOperator{\Supp}{{Supp}}

\DeclareMathOperator{\Hom}{Hom}

\DeclareMathOperator{\Ker}{{Ker}}
\DeclareMathOperator{\Coker}{{Coker}}
\DeclareMathOperator{\Ass}{{Ass}}

\newcommand{\calH}{\mathcal{H}}
\newcommand{\calK}{\mathcal{K}}
\newcommand{\calL}{\mathcal{L}}
\newcommand{\calM}{\mathcal{M}}
\newcommand{\calN}{\mathcal{N}}

\newcommand{\fm}{\mathfrak{m}}

\newcommand{\fa}{\mathfrak{a}}

\input{xy}
\xyoption{all}
\usepackage{tikz}
\newcommand{\FF}{\mathbb{F}}

\DeclareMathOperator{\Ann}{{Ann}}

\usepackage{fullpage}

\usepackage{setspace}
\usepackage{hyperref}

\usepackage{enumerate}

\usepackage{graphicx}

\usepackage[all,cmtip]{xy}

\usepackage{verbatim}

\makeatletter
\@namedef{subjclassname@2020}{%
  \textup{2020} Mathematics Subject Classification}
\makeatother

 \title{$F^e$-modules with applications to $D$-modules}
\author{Wenliang Zhang}\address{Department of Mathematics, Statistics, and Computer Science, University of Illinois at Chicago, Chicago, IL 60607}
\email{wlzhang@uic.edu}
\thanks{The author is partially supported by NSF through DMS-1752081.}

\begin{document}
\maketitle

\begin{abstract}
Using a theory of $F^e$-modules (a natural extension of Lyubeznik's $F$-module theory), we extend results on Matlis dual of $F$-finite $F$-modules to $D$-submodules of $F^e$-finite $F^e$-modules and apply these results to address the Lyubeznik-Yildirim conjecture in mixed characteristic.
\end{abstract}

\section{Introduction}
The theory of $F$-modules, whose roots can be found in \cite{PeskineSzpiroDimensionProjective, HartshorneSpeiserLocalCohomologyInCharacteristicP, HunekeSharp}, is introduced in \cite{LyubeznikFModules}. Since its introduction, it has been proven indispensable in the study of rings of prime characteristic $p$ (see, for instance, \cite{EmertonKisinRH, MontanerBlickleLyubeznik2005, LyubeznikVanishingLCCharp, ZhangLyubeznikNumbersProjectiveSchemes, LyubeznikSinghWalther}). Replacing the Peskine-Szpiro functor $F(-)$ by its $e$-th iteration $F^e(-)$, one obtains the theory of $F^e(-)$-modules (details can be found in \S\ref{background}). 

The motivations behind this article are two-fold:
\begin{enumerate}
\item[(\dag)] It follows from \cite[7.4]{KatzmanMaSmirnovZhang2018} that there exists an $F$-finite $F$-module which admits a simple $D$-submodule that is {\it not} an $F$-submodule ({\it cf.} Example \ref{F^3 not F}). The structure of such $D$-submodules warrants further investigations.
\item[(\ddag)] Let $R$ be the completion of $\ZZ[x_1,\dots ,x_6]$ at the maximal ideal $(2,x_1,\dots,x_6)$ and let $I$ be the monomial ideal associated with the minimal triangulation of the projective plane. Then it is proved in \cite[4.5]{DattaSwitalaZhang} that the support of the Matlis dual of $H^4_I(R)$ is $\Spec(R/(2))$, a proper subset of $\Spec(R)$, which provides a counterexample to \cite[Conjecture 1]{LyubeznikYildirim}. It is natural to ask whether $\Spec(R/(\pi))$ is always contained in the support of the Matlis dual of $H^j_I(R)$ whenever $R=V[[x_1,\dots,x_n]]$ and $V=(V,\pi V)$ is a complete DVR of mixed characteristic $(0,p)$. 
\end{enumerate}

One of our results regarding $(\dag)$ is the following. 
\begin{theorem}
\label{simple D submodule}
Let $R$ be a noetherian regular ring of characteristic $p$ which is a finitely generated $R^p$-module. If $\calN$ is a simple $D_R$-submodule of an $F^e$-finite $F^e$-module $\calM$, then there exists a positive integer $e'$ such that $\calN$ is an $F^{e'e}$-submodule of $\calM$.
\end{theorem}

Since each $F^e$-module is naturally a $D_R$-module (Remark \ref{D-structure on F-module}), it is feasible to consider $D$-submodules of an $F^e$-module in the statement of Theorem \ref{simple D submodule}. Example \ref{F^3 not F} shows that, in general, it is necessary to have $e'>1$, even when $e=1$. This provides one of the justifications for the necessity of considering $F^e$-modules (with $e>1$).

As a consequence of our Theorem \ref{simple D submodule}, we have the following result concerning $(\ddag)$. 
\begin{theorem}
\label{support D-submodule}
Let $(R,\fm)$ be a noetherian regular local ring of finite type over a regular local ring $A$ such that $A$ is module-finite over $A^p$. Let $\calN$ be an arbitrary (not necessarily simple) $D_R$-submodule of an $F^e$-finite $F^e$-module. Assume that $(0)$ is not an associated prime of $\calN$. Then
\[\Supp_R(\calN^\vee)=\Spec(R),\]
where $\calN^\vee$ denotes the Matlis dual of $\calN$.
 \end{theorem}

 \begin{theorem}
 \label{support in mixed}
 Let $R=V[[x_1,\dots,x_d]]$ be a formal power series ring over a complete DVR $(V,\pi V,k)$ of mixed characteristic and $I$ be an ideal of $R$ such that $I\not\subseteq (\pi)$. Assume that $[k:k^p]<\infty$. If 
\begin{enumerate}
\item either $\Coker(H^j_I(R)\xrightarrow{\pi}H^j_I(R))\neq 0$
\item or $\Ker(H^j_I(R)\xrightarrow{\pi}H^j_I(R))\neq 0$,
\end{enumerate}
then 
\[\Spec(R/\pi R)\subseteq \Supp_R(H^j_I(R)^\vee),\]
where $H^j_I(R)^\vee$ denotes the Matlis dual of $H^j_I(R)$. 
 \end{theorem}
 
Theorem \ref{support D-submodule} is a natural extension of the main theorem in \cite{LyubeznikYildirim}. Without any further assumptions on $I$, Theorem \ref{support in mixed} is the best possible since ``$\subseteq$" can be ``$=$" in general; see Remark \ref{why best in mixed} for details.

This article is organized as follows. In \S\ref{background}, we collect some necessary preliminaries on $F^e$-modules and $D$-modules; in \S\ref{F and D}, we prove Theorem \ref{simple D submodule} and its corollary; in \S\ref{Matlis dual}, we apply results proved in \S\ref{F and D} to the investigation of the support of Matlis dual of $D$-modules, especially local cohomology modules.

\subsection*{Acknowledgment} The author thanks the referee for carefully reading the article and for the comments which improve the exposition of the article.

\section{Background and some results on $D$-modules and $F^e$-modules}
\label{background}
Let $A$ be a commutative ring with identity. A ($\ZZ$-linear) differential operator of order 0 is the multiplication by an element of $A$. A differential operator of order $\leq \ell$ is an additive map $\delta:A\to A$ such that the commutator $[\delta,\tilde{a}]=\delta\circ \tilde{a}-\tilde{a}\circ \delta$ is a differential operator of order $\leq \ell-1$, where $\tilde{a}:A\to A$ is the multiplication by $a\in A$, for every $a\in A$. These differential operators form a ring, denoted by $D_{A|\ZZ}$ or simply $D_A$. 

If $k\subseteq A$ is a subring, then the ring of $k$-linear differential operators, denoted by $D_{A|k}$, is the subring of $D_A$ consisting of $k$-linear elements of $D_A$. Given any element $f\in A$, $A_f$ carries a natural $D_{A|k}$-module structure. Consequently, the local cohomology modules $H^j_{\fa}(A)$ carry a natural $D_{A|k}$-module structure for each ideal $\fa$ in $A$.

Assume now that $A$ contains a field of characteristic $p$, and let $A^{p^e}$ be the subring of $A$ consisting of all the $p^e$-th powers of all elements in $A$ for each positive integer $e$.Then, every differential operator $\delta\in D_A$ of order $\leq p^e-1$ is $A^{p^e}$-linear; that is $\delta\in \Hom_{A^{p^e}}(A,A)$. Let $k$ be a perfect subfield of $A$ ({\it e.g.} $k=\ZZ/p\ZZ$). Assume that $A$ is a finite $k[A^p]$-module, then
\[D_{A}=D_{A|k}=\bigcup_e\Hom_{A^{p^e}}(A,A).\]
$\Hom_{A^{p^e}}(A,A)$ is also denoted by $D^{(e)}$ in the literature.

When $A=B[x_1,\dots,x_n]$ or $A=B[[x_1,\dots,x_n]]$ where $B$ is a commutative ring with identity, the ring $D_{A|B}$ can be described explicitly as follows. Set $\partial^{[t]}_i:=\frac{1}{t!}\frac{\partial^t}{\partial x^t_i}$; that is
\[\partial^{[t]}_i(x^s_i)=\begin{cases} 0& s<t\\ \binom{s}{t}x^{s-t}_i&s\geq t \end{cases}\]
Then $D_{A|B}$ is the ring extension of $A$ generated by $\partial^{[t]}_i$ for all $i$ and all $t\geq 1$. Furthermore, if $B$ is a perfect field of characteristic $p$, then $D^{(e)}$ is the ring extension of $A$ generated by $\partial^{[t]}_i$ for all $i$ and all $t\leq p^e-1$.

\begin{remark}
\label{kernel and cokernel}
Given this explicit descriptions of the rings of differential operators, one can check the following ({\it cf.} \cite[2.1]{BBLSZ1} for details). Assume that $R=V[x_1,\dots,x_n]$ or $R=V[[x_1,\dots,x_n]]$ where $V=(V,\pi V,k)$ is a DVR with a uniformizer $\pi$. Set $\bar{R}=R/(\pi)$. Given each $D_{R|V}$-module $M$, the multiplication map $M\xrightarrow{\pi}M$ is $D_{R|V}$-linear as $\pi\in V$; consequently, the submodule $\Ann_M(\pi)$ and the quotient module $M/\pi M$ are naturally $D_{R|V}$-modules and $D_{\bar{R}|k}$-modules. The short exact sequence $0\to R\xrightarrow{\pi}R\to \bar{R}\to 0$ induces a long exact of local cohomology modules
\[\cdots \to H^{j-1}_I(\bar{R})\to H^j_I(R)\xrightarrow{\pi}H^j_I(R)\to H^{j}_I(\bar{R})\to \cdots\]
which is an exact sequence in the category of $D_{R|V}$-modules, for each ideal $I$ of $R$. In particular, the modules $\Coker(H^j_I(R)\xrightarrow{\pi}H^j_I(R))$ and $\Ker(H^j_I(R)\xrightarrow{\pi}H^j_I(R))$ are naturally $D_{R|V}$-modules and $D_{\bar{R}|k}$-modules. Consequently, the natural maps
\[\Coker(H^j_I(R)\xrightarrow{\pi}H^j_I(R))\to H^{j}_I(\bar{R}),\ H^{j-1}_I(\bar{R})\to \Ker(H^j_I(R)\xrightarrow{\pi}H^j_I(R))\]
are morphisms in the category of $D_{R|V}$-modules and morphisms in the category of $D_{\bar{R}|k}$-modules.
\end{remark}

Let $A$ be a commutative ring that contains a field of characteristic $p$. Then there is a natural functor on the category of $A$-modules called the Peskine-Szpiro functor and defined as follows. Let $F^e_*A$ denote the $A$-module whose underlying abelian group is the same as $A$ and whose $A$-module structure is induced by the $e$-th Frobenius $A\xrightarrow{a\mapsto a^{p^e}}A$. The Psekine-Szpiro functor $F^e_A(-)$ on the category of $A$-modules is defined by
\[F^e_A(M)=F^e_*A\otimes_AM.\]

\begin{remark}
\label{equivalence between De and R}
Assume $R$ is a noetherian regular ring of characteristic $p$. Then a classical theorem due to Kunz (\cite{KunzCharacterizationsOfRegularLocalRings}) asserts that the Peskine-Szpiro functor $F^e_R$ is an exact functor.

Moreover, assume that $R$ is a finite generated $R^p$-module. Then the category of $R$-modules is equivalent to the category of $D^{(e)}$-modules (\cite[Proposition 2.1]{MontanerBlickleLyubeznik2005}). The functor from the category of $R$-modules to the category of $D^{(e)}$-modules is precisely the Peskine-Szpiro functor $F^e_R$. Since one can identify $\Hom_{R^{p^e}}(R,R)$ with $\Hom_R(F^e_*R,F^e_*R)$, the $D^{(e)}$-module structure on $F^e_R(M)=F^e_*R\otimes_RM$ is induced by the action on $F^e_*R$. We refer the reader to \cite{MontanerBlickleLyubeznik2005} for details.
\end{remark}

The following result, \cite[Proposition 2.3]{MontanerBlickleLyubeznik2005}, will be useful in the sequel.
\begin{theorem}
\label{F-equiv on D-modules}
Let $R$ be a noetherian regular ring of characteristic $p$. Assume that $R$ is a finitely generated $R^p$-module. Then the Peskine-Szpiro functor $F^e_R$ is an equivalence of the category of $D_R$-modules with itself.
\end{theorem}

Considerations on the Peskine-Szpiro functor have proven to be fruitful in the investigation of rings of prime characteristic $p$. When $R$ is regular, the theory of $F$-modules is introduced in \cite{LyubeznikFModules}. Since this theory is readily adapted to the $e$-th Peskine-Szpiro functor $F^e(-)$, we opt to explain here the theory of $F^e$-modules. 

For the rest of the section, $R$ denotes a noetherian regular ring of prime characteristic $p$.

\begin{definition} Let $e$ be a positive integer. 
\begin{enumerate}
\item An $R$-module $\calM$ is an \emph{$F^e$-module} if there is an $R$-module isomorphism
\[\theta: \calM\to F^e(\calM)=F^e_*R\otimes_R\calM\]
called the structure isomorphism.

When $e=1$, we will write $F$ instead of $F^{1}$ whenever the context is clear.

\item If $(\calM,\theta_\calM)$ and $(\calN,\theta_\calN)$ are $F^e$-modules, then an \emph{$F^e$-module morphism} from $(\calM,\theta_\calM)$ to $(\calN,\theta_\calN)$ consists of the the following commutative diagram:
\[\xymatrix{
\calM \ar[r]^{\varphi} \ar[d]^{\theta_\calM} & \calN\ar[d]^{\theta_\calN}\\
F^e(\calM) \ar[r]^{F^e(\varphi)} & F^e(\calN)
}\]
We will simply write this $F^e$-module morphism as $\varphi:\calM\to \calN$ whenever the context is clear.

\item A \emph{generating morphism} of an $F^{e}$-module is an $R$-module homomorphism $\beta: M\to F^e(M)$, where $M$ is an $R$-module, such that $\calM$ is the direct limit of the top row of the following commutative diagram
\[\xymatrix{
M \ar[r]^{\beta} \ar[d]^{\beta} & F^e(M) \ar[r]^{F^e(\beta)} \ar[d]^{F^e(\beta)} & F^{2e}(M) \ar[r] \ar[d] &\cdots\\
F^e(M) \ar[r]^{F^e(\beta)} & F^{2e}(M) \ar[r]^{F^{2e}(\beta)} & F^{3e}(M)\ar[r] & \cdots
}\]
and the structure isomorphism $\theta:\calM\to F^e(\calM)$ is induced by the vertical morphism in the diagram.

\item An $F^e$-module $\calM$ is \emph{$F^e$-finite} if it admits a generating morphism $\beta:M\to F^e(\calM)$ where $M$ is a finitely generated $R$-module.
\end{enumerate}
We will denote the category of $F^e$-modules by $\sF^e$.
\end{definition}

Results on $F$-modules in the literature, {\it e.g.} \cite{LyubeznikFModules} and \cite{MontanerBlickleLyubeznik2005}, can be readily extended to $F^e$-modules by simply replacing the functor $F(-)$ with the functor $F^e(-)$. Before proceeding to properties of $F^e$-modules, we would like to explain one of the motivations behind introducing these modules and hopefully to answer the natural question: why not just work with $F$-modules?

\begin{example}
\label{F^3 not F}
Let $R=\FF_{11}[x,y,z]$ and let $f=x^7+y^7+z^7$. Denote $H^1_{(f)}(R)$ by $\calH$. Then \cite[7.4]{KatzmanMaSmirnovZhang2018} shows that
\[\ell_{\sD}(\calH)>\ell_{\sF}(\calH),\]
where $\ell_{\sD}(\calH)$ (or $\ell_{\sF}(\calH)$, respectively) denotes the length of $\calH$ in the category of $\sD$-modules (or in $\sF$, respectively). Let $H$ be a simple $D$-submodule of $\calH$. If $H$ is an $F$-submodule of $\calH$, then it follows from \cite[Theorem~2.8]{LyubeznikFModules} that $H$ is an $F$-finite $F$-submodule of $\calH$ and consequently $\calH/H$ is an $F$-finite $F$-module. Note $\ell_{\sD}(\calH/H)>\ell_{\sF}(\calH/H)$. Continuing this process, after at most $\ell_{\sF}(\calH)$ steps, one can see that there is an $F$-finite $F$-module $\calH'$ (a quotient of $\calH$ in the category of $F$-modules) such that $\calH'$ admits a simple $D$-submodule that is not an $F$-submodule of $\calH'$. (Similarly, one can also deduce that $\calH$ admits a $\sD$-submodule which is not an $F$-submodule.) 
\end{example}

Example \ref{F^3 not F} shows that the theory of $F^e$-modules may be applicable to $\sD$-submodules of an $F$-finite $F$-module which may not be $F$-submodules in general.

\begin{remark}
Assume that $(\calM,\theta)$ is an $F^e$-module for a positive integer $e$. Then, for every positive integer $t$, the composition
\[\calM\xrightarrow{\theta}F^e(\calM)\xrightarrow{F^e(\theta)}\cdots\to F^{te}(\calM)\]
is also an $R$-module isomorphism. Hence $\calM$ is an $F^{te}$-module for every positive integer $t$. In particular, an $F$-module is also an $F^e$-module for every positive integer $e$. Consequently, all local cohomology modules $H^j_I(R)$ (and iterated local cohomology modules) are $F^e$-modules for every positive integer $e$.

Assume that $(\calM,\theta)$ is an $F^e$-module for a positive integer $e$. Then so is $F^t(\calM)$ for every positive integer $t$ since $F^t(\calM)\cong F^t(F^e(\calM))=F^e(F^t(\calM))$. 

Let $e,f$ be positive integers such that $e|f$. Then $\sF^e$ can be naturally viewed as a subcategory of $\sF^{f}$. Let $\calM$ be an $F^e$-module. By an $F^f$-submodule $\calN$ of $\calM$ we mean a sub-object of $\calM$ when $\calM$ is viewed as an object in $\sF^f$.
\end{remark}

\begin{remark}
\label{D-structure on F-module}
Every $F^e$-module admits a natural $D$-module structure. This follows from Remark \ref{equivalence between De and R}.
Let $\delta$ be a differential operator. Then there exists an positive integer $t$ such that its order (as a differential operator) is less than $te$. Let $\alpha_t$ denote the composition 
\[\calM\xrightarrow{\theta}F^e(\calM)\xrightarrow{F^e(\theta)}\cdots\to F^{te}(\calM)\]

Given an arbitrary element $m\in \calM$, write $\alpha_t(m)=\sum_ir_i\otimes m_i$. Then, for every element $m\in \calM$, set \[\delta\cdot m:=\alpha_t^{-1}(\sum_i(\delta\cdot r_i)\otimes m_i)\]

Whenever we view an $F^e$-module as a $D_R$-module, we always refer to the $D_R$-module structure specified in the previous paragraph. Under this $D_R$-module structure, an $F^e$-module morphism between any two $F^e$-modules is also a $D_R$-module morphism.
\end{remark}

We now collect some results on $F^e$-modules which are natural analogues of corresponding results on $F$-modules in the literature. 
\begin{remark}
\label{D-results on Fe}
Let $R$ be a noetherian regular ring of characteristic $p$ that is module-finite over $R^p$. Let $\calM$ be an $F^e$-module for a positive integer $e$.
\begin{enumerate}
\item The $F^e$-finite modules form a full abelian subcategory of the category of $F^e$-modules which is closed under formation of submodules, quotient modules and extensions. When $e=1$, this is \cite[Theorem 2.8]{LyubeznikFModules}. When $e$ is an arbitrary positive integer, the same proof goes through (by replacing $F(-)$ with $F^e(-)$).
 
\item The structure isomorphism $\theta:\calM\to F^e(\calM)$ is $D_R$-linear, where the $D_R$-module structure is as in Remark \ref{D-structure on F-module}. When $e=1$, this is \cite[Lemma 2.4]{MontanerBlickleLyubeznik2005}. When $e$ is an arbitrary positive integer, the same proof goes through (by replacing $F(-)$ with $F^e(-)$).

\item Assume further that $R$ is of finite type over a regular local ring $A$ such that $A$ is module-finite over $A^p$. Then every $F^e$-finite $F^e$-module has finite length in $\sF^e$ and in the category of $D_R$-modules, for each positive integer $e$. When $e=1$, this is \cite[Theorem 3.2]{LyubeznikFModules} and \cite[Theorem 2.5]{MontanerBlickleLyubeznik2005}, respectively. When $e$ is an arbitrary positive integer, the same proofs go through (by replacing $F(-)$ with $F^e(-)$).
\end{enumerate}
\end{remark}

\section{Interactions between $F^e$-modules and $D$-modules}
\label{F and D}
The main goal of this section is to prove Theorem \ref{simple D submodule}. We begin with the following observation.

\begin{proposition}
\label{EK remark}
Let $R$ be a noetherian regular ring of finite type over a regular local ring $A$ such that $A$ is module-finite over $A^p$ and let $\calM$ be an $F^e$-finite $F^e$-module. Let $\calN$ be a $D_R$-submodule of $\calM$. Assume that\footnote{Here we identify $F^e(\calN)$ with an $R$-submodule of $\calM$ under the isomorphism $F^e(\calM)\cong \calM$.} $F^e(\calN)\subseteq \calN$. Then $\calN$ is an $F^e$-submodule of $\calM$.
\end{proposition}
\begin{proof}
Since $F^e(\calN)$ is naturally a $D_R$-submodule of $\calM$ (due to Theorem \ref{F-equiv on D-modules}) and $F^e(\calN)\subseteq \calN$, we have a descending chain of $D_R$-submodules of $\calM$:
\[\calN\supseteq F^e(\calN)\supseteq F^{2e}(\calN)\supseteq \cdots.\]
Since $\calM$ has finite length in the category of $D_R$-modules (Remark \ref{D-results on Fe}), this chain must terminate in finitely many steps; that is $F^{te}(\calN)=F^{(t+1)e}(\calN)=F^{te}(F^e(\calN))$ for an integer $t$. Hence $\calN\cong F^e(\calN)$ which completes the proof.
\end{proof}

We are now  in position to prove Theorem \ref{simple D submodule}, whose proof is inspired by the proof of \cite[Theorem 5.6]{LyubeznikFModules}.

\begin{proof}[Proof of Theorem \ref{simple D submodule}]
Since $\calM$ is an $F^e$-module, $F^{te}(\calN)\subseteq F^{te}(\calM)\cong \calM$ for each positive integer $t$. We will view $F^{te}(\calN)$ as a $D$-submodule of $\calM$. It follows from Theorem \ref{F-equiv on D-modules} that $F^{te}(\calN)$ is also a simple $D$-submodule of $\calM$ for every positive integer $t$. Let $t$ be the least positive integer such that 
\[\calN+F^e(\calN)+\cdots+F^{(t-1)e}(\calN)=\calN\oplus F^e(\calN)\oplus \cdots \oplus F^{(t-1)e}(\calN)\] 
that is, $t$ is the least positive integer such that
\[(\calN+F^e(\calN)+\cdots+F^{(t-1)e}(\calN))\cap F^{te}(\calN)\neq \emptyset.\] 
Set \[\calL:=\calN+F^e(\calN)+\cdots+F^{(t-1)e}(\calN)=\calN\oplus F^e(\calN)\oplus \cdots \oplus F^{(t-1)e}(\calN).\] 
By the construction of $\calL$, one sees that $\calL$ is a semi-simple $D_R$-module.

We claim that $\calL$ is an $F^e$-submodule of $\calM$ and we reason as follows. Since $F^{te}(\calN)$ is also a simple $D_R$-module by Theorem \ref{F-equiv on D-modules} and $F^{te}(\calN)\cap \calL\neq \emptyset$, we have 
\[F^{te}(\calN)\subset \calL.\]
Consequently $F^e(\calL)\subseteq \calL$. It follows from Proposition \ref{EK remark} that $\calL$ is an $F^e$-submodule of $\calM$ and hence is also an $F^e$-finite $F$-finite module by Remark \ref{D-results on Fe}.

This shows that $\calN$ is a simple $D_R$-submodule of an $F^e$-finite $F^e$-module $\calL$ such that $\calL=\calN\oplus \cdots \oplus F^{(t-1)e}(\calN)$, where $\calN, \dots, F^{(t-1)e}(\calN)$ are simple $D$-submodules of $\calL$. Since $\calL$ is a semi-simple $D_R$-module stable under $F^e(-)$ and $F^e(-)$ is an equivalence on the category of $D$-modules (Theorem \ref{F-equiv on D-modules}), the functor $F^e(-)$ cycles through its direct summands $\calN, \dots, F^{(t-1)e}(\calN)$. Therefore, there exists a positive integer $e'$ such that $\calN\cong F^{e'e}(\calN)$. This finishes the proof.
\end{proof}


\begin{corollary}
\label{quotient is Fe-module}
Let $R$ be as in Theorem \ref{simple D submodule}. Assume that $\calN$ is a $D_R$-module quotient of an $F^e$-finite $F^e$-module $\calM$. Then there exists a positive integer $e'$ such that $\calN$ is an $F^{e'}$-finite $F^{e'}$-module. 
\end{corollary}
\begin{proof}
We will use induction on the length of $\calM$ as a $D_R$-module; note that $\calM$ has finite length in the category of $D_R$-modules according to Remark \ref{D-results on Fe}.

When $\calM$ is a simple $D_R$-module, then either $\calN=0$ or $\calN=\calM$. The conclusion is clear.

Let $\ell_{\sD}(\calM)$ denote the $D_R$-module length of $\calM$. Assume now $\ell_{\sD}(\calM)\geq 2$ and the theorem has been proved for all $F$-finite $F$-modules with $D_R$-module length $\leq \ell_{\sD}(\calM)-1$. Since $\calN$ is a $D_R$-module quotient, there is a $D_R$-submodule $\calL$ of $\calM$ such that $\calN=\calM/\calL$. Since $\ell_{\sD}(\calL)<\infty$, there is a simple $D_R$-submodule $\calL'$ of $\calL$. Since $\calL'$ is a simple $D_R$-submodule of $\calM$,  by Theorem \ref{simple D submodule} $\calL'$ is an $F^{te}$-submodule of $\calM$ for a positive integer $t$. Consequently, $\calM/\calL'$ is an $F^{te}$-finite $F^{te}$-finite module. Set $\bar{\calM}:=\calM/\calL'$ and $\bar{\calL}:=\calL/\calL'$. Since $\ell_{\sD}(\bar{\calM})<\ell_{\sD}(\calM)$, by induction $\bar{\calM}/\bar{\calL}$ is an $F^{e'}$-finite $F^{e'}$-module for a positive integer $e'$. Since $\calN=\calM/\calL\cong \bar{\calM}/\bar{\calL}$, this completes the proof.
\end{proof}

\begin{remark}
When $R=k[x_1,\dots,x_n]$ where $k$ is a field of characteristic $p$, one can also develop the notions of graded $F^e$-modules and graded $F^e$-finite $F^e$-modules and to extend results on graded $F$-modules to graded $F^e$-modules. For instance, one can show that a graded $F^e$-finite $F^e$-module is also an Eularian graded $D_R$-module; the interested reader is referred to \cite{EulerianDModules} for the notion of Eulerian graded $D$-modules. We opt not to pursue this in the current article.
\end{remark}

\section{Applications to Matlis dual}
\label{Matlis dual}
Prompted by the work of Hellus in \cite{Hellus}, Lyubeznik and Yildirim conjectured (in \cite[Conjecture 1]{LyubeznikYildirim}) that, if $R$ is a noetherian regular local ring and $H^j_I(R)\neq 0$ where $I$ is an ideal of $R$, then $\Supp_R(H^j_I(R)^\vee)=\Spec(R)$. Here $H^j_I(R)^\vee$ denotes the Matlis dual of $H^j_I(R)$. This conjecture is proved in \cite{LyubeznikYildirim} in characteristic $p$. In mixed characteristic, this conjecture is shown to be false as stated (\cite[4.5]{DattaSwitalaZhang}). One may notice that the example in \cite[4.5]{DattaSwitalaZhang} ({\it cf.} Remark \ref{why best in mixed}) is the kernel of multiplication by the uniformizer $\pi$ on a local cohomology module $H^j_I(R)$; where $(R,V)$ is a formal power series ring over a complete discrete valuation ring $V$ with a uniformizer $\pi$. The main purpose of this section is to prove Theorem \ref{support in mixed} which generalizes \cite[4.5]{DattaSwitalaZhang}.

We begin with the following extension of \cite[Theorem 1.1]{LyubeznikYildirim}.
\begin{theorem}
\label{LY theorem for Fe}
Let $(R,\fm)$ be a complete regular local ring of characteristic $p$ and let $\calM$ be an $F^e$-finite $F^e$-module for a positive integer $e$. Assume that $(0)\notin \Ass_R(\calM)$. Then
\[\Supp(\Hom_R(\calM,E_R(R/\fm)))=\Spec(R),\]
where $E_R(R/\fm)$ denotes the injective hull of $R/\fm$. 
\end{theorem}
\begin{proof}
Once one replaces the functor $F(-)$ by $F^e(-)$, the proof is the same as the one of \cite[Theorem 1.1]{LyubeznikYildirim}. To avoid duplication, we opt not to repeat the details here.
\end{proof}

For each $R$-module $M$, we will denote its Matlis dual, $\Hom_R(M,E_R(R/\fm))$, by $M^\vee$.

We now can prove our Theorem \ref{support D-submodule} which is a natural extension of Theorem \ref{LY theorem for Fe} to $D_R$-submodules of $F^e$-finite $F^e$-modules.

\begin{proof}[Proof of Theorem \ref{support D-submodule}]
Let $\calL$ be a simple $D_R$-submodule of $\calN$. Since $\Ass_R(\calL)\subseteq \Ass_R(\calN)$ and $(0)\notin \Ass_R(\calN)$, we have $(0)\notin \Ass_R(\calL)$. The short exact sequence $0\to \calL\to \calN\to \calN/\calL\to 0$ induces a short exact sequence 
\[0\to (\calN/\calL)^\vee\to \calN^\vee \to \calL^\vee\to 0.\]
If $\Supp(\calL^\vee)=\Spec(R)$, then it follows from the short exact sequence above that $\Supp(\calN^\vee)=\Spec(R)$. We are now reduced to proving that $\Supp(\calL^\vee)=\Spec(R)$.

Since $\calL$ is a simple $D_R$-submodule of an $F^e$-finite $F^e$-module, it follows form Theorem \ref{simple D submodule} that $\calL$ is an $F^{e'e}$-submodule of $\calM$. Since $\calM$ is $F^e$-finite (hence $F^{e'e}$-finite) and $\calL$ is an $F^{e'e}$-submodule, $\calL$ must be $F^{e'e}$-finite as well because of Remark \ref{D-results on Fe}. Now Theorem \ref{LY theorem for Fe} finishes the proof.
\end{proof}

We now apply Theorem \ref{support D-submodule} to prove Theorem \ref{support in mixed}. 

\begin{proof}[Proof of Theorem \ref{support in mixed}]
The short exact sequence $0\to R\xrightarrow{\pi} R\to R/\pi R\to 0$ induces a long exact sequence
\[\cdots\to H^{j-1}_I(\bar{R})\to H^j_I(R)\xrightarrow{\pi} H^j_I(R)\to H^j_I(\bar{R})\to \cdots\]
which implies 
\begin{enumerate}
\item an injection $\Coker\Big(H^j_I(R)\xrightarrow{\pi}H^j_I(R)\Big)\hookrightarrow H^j_I(\bar{R})$, and
\item a surjection $H^{j-1}_I(\bar{R})\to \Ker\Big(H^j_I(R)\xrightarrow{\pi}H^j_I(R)\Big)$.
\end{enumerate} 

Note that $\Coker(H^j_I(R)\xrightarrow{\pi}H^j_I(R)), \Ker(H^j_I(R)\xrightarrow{\pi}H^j_I(R)), H^j_I(\bar{R}), H^{j-1}_I(\bar{R})$ carry a natural $D_{\bar{R}|k}$-module structure, and that both the injection and the surjection are $D_{\bar{R}|k}$-linear. ({\it cf.} Remark \ref{kernel and cokernel}). This makes 
\begin{enumerate}
\item $\Coker(H^j_I(R)\xrightarrow{\pi}H^j_I(R))$ a $D_{\bar{R}|k}$-submodule of $H^j_I(\bar{R})$ which is an $F_{\bar{R}}$-finite $F_{\bar{R}}$-module, and
\item $\Ker(H^j_I(R)\xrightarrow{\pi}H^j_I(R))$ a $D_{\bar{R}|k}$-module quotient of $H^{j-1}_I(\bar{R})$ which is an $F_{\bar{R}}$-finite $F_{\bar{R}}$-module.
\end{enumerate} 

Assume $\Coker(H^j_I(R)\xrightarrow{\pi}H^j_I(R))=H^j_I(R)/\pi H^j_I(R)\neq 0$. Then it follows that $H^j_I(\bar{R})\neq 0$ (the image of $I$ in $\bar{R}$ is a nonzero ideal). Hence $H^j_I(R)/\pi H^j_I(R)$ satisfies the assumptions in Theorem \ref{support D-submodule}; our assumption on $k$ ensures that $\bar{R}$ satisfies the hypothesis in Theorem \ref{support D-submodule}. Consequently 
\[\Spec(\bar{R})=\Supp_{\bar{R}}\Hom_{\bar{R}}(H^j_I(R)/\pi H^j_I(R), \bar{E}),\]
where $\bar{E}$ denotes the injective hull of $k$ as an $\bar{R}$-module.

Since $\bar{E}\cong \Hom_R(\bar{R},E)$ where $E=E_R(k)$, by the adjunction between $\otimes$ and $\Hom$, we have
\[\Hom_{\bar{R}}(H^j_I(R)/\pi H^j_I(R), \bar{E})\cong \Hom_R(H^j_I(R)/\pi H^j_I(R),E).\] 
It follows that 
\[\Spec(\bar{R})=\Supp_{R}\Hom_{R}(H^j_I(R)/\pi H^j_I(R), E)\]
where $\Spec(\bar{R})$ is considered a closed subset of $\Spec(R)$. 

The surjection $H^j_I(R)\to H^j_I(R)/\pi H^j_I(R)$ induces an injection $(H^j_I(R)/\pi H^j_I(R))^{\vee}\hookrightarrow H^j_I(R)^{\vee}$. Therefore,
\[\Spec(\bar{R})\subseteq \Supp_R(H^j_I(R)^{\vee}).\]

Assume $\Ker(H^j_I(R)\xrightarrow{\pi}H^j_I(R))\neq 0$ and set $\calK:=\Ker(H^j_I(R)\xrightarrow{\pi}H^j_I(R))$. Then $\calK$ is a $D_{\bar{R}|k}$-module quotient of $H^{j-1}_I(\bar{R})$ and $H^{j-1}_I(\bar{R})$ is an $F_{\bar{R}}$-finite $F_{\bar{R}}$-module. According to Corollary \ref{quotient is Fe-module}, $\calK$ is an $F^{e''}$-finite $F^{e''}$-module and consequently 
\[\Supp_R(\calK^\vee)=\Spec(\bar{R})\]
by Theorem \ref{LY theorem for Fe}. The injection $\calK\hookrightarrow H^j_I(R)$ induces a surjection $H^j_I(R)^\vee\twoheadrightarrow \calK^\vee$ which proves that $\Supp_R(\calK^\vee)\subseteq \Supp_R(H^j_I(R)^\vee)$. This completes the proof.
\end{proof}

\begin{remark}
Let $(R,\fm,k)$ be a complete unramified regular local ring of mixed characteristic and $I$ be an arbitrary nonzero ideal. Let $(V,\pi V,k)$ be its coefficient DVR. A $D_{R|V}$-submodule $M$ of a local cohomology module $H^j_I(R)$ may not satisfy the conclusion in Theorem \ref{support in mixed}. For instance, set $I=(\pi)$ and $M:=\Ann_{H^1_I(R)}(\pi)$. Then one can check 
\begin{enumerate}
\item $M\cong R/(\pi)$
\item $M$ is a $D_{R|V}$-submodule of $H^1_I(R)$
\end{enumerate}
Hence $\Supp_R(M^\vee)=\{\fm\}$.

This also indicates that the assumption ``$I\not\subseteq(\pi)$" in Theorem \ref{support in mixed} is necessary for its proof.
\end{remark}

\begin{remark}
\label{why best in mixed}
Let $R$ be the completion of $\ZZ[x_1,\dots,x_6]$ at the maximal ideal $\fm=(2,x_1,\dots,x_6)$. Let $I$ the monomial ideal associated with the minimal triangulation of the real projective plane. It is proved in \cite[4.5]{DattaSwitalaZhang} that
\[H^4_I(R)\cong \Hom_R(R/(2), H^7_{\fm}(R))\]
and consequently $\Supp_R(H^4_I(R)^\vee)=\Spec(R/(2))$. 

Therefore, without any further assumptions, the conclusion in Theorem \ref{support in mixed} is the best possible.
\end{remark}

In light of Theorem \ref{support in mixed}, we would like to ask the following.
\begin{question}
Let $R$ be a complete unramified regular local ring of mixed characteristic and $(V,\pi V)$ be its coefficient ring. 
\begin{enumerate}
\item Is it always true that
\[\Spec(R/\pi R)\subseteq \Supp_R(H^j_I(R)^\vee)\]
for each ideal $I$ and each integer $j$?
\item Can one characterize the local cohomology modules $H^j_I(R)$ such that 
\[\Supp_R(H^j_I(R)^\vee)=\Spec(R)?\]
\end{enumerate}
\end{question}

\bibliographystyle{skalpha}
\bibliography{CommonBib}
\end{document}